\documentclass[12pt]{article}
\usepackage{graphicx,amsfonts,amsmath,enumitem,tikz}
\usepackage{amssymb}
\usepackage{amscd}
\usepackage{comment}
\usepackage{verbatim}
\usetikzlibrary{arrows,automata,shapes,snakes}
\usepackage{amsfonts}
\usepackage{setspace}
\usepackage{mathtools}

\DeclarePairedDelimiter{\ceil}{\lceil}{\rceil}

\def\var{{\rm Var}}
\def\cov{{\rm Cov}}
\def\ba{{\bf a}}
\def\dtv{{d_{\rm TV}}}
\def\calL{{\cal L}}
\def\Po{{\rm Po}}

\def\var{{\rm Var}}

\def\reel{\hbox{{\rm R}\kern-1em\hbox{{\rm I} }}}
\def\relatif{\ \hbox{{\rm Z}\kern-.4em\hbox{\rm Z}}}
\def\nat{\hbox{{\rm N}\kern-1em\hbox{{\rm I} } }}
\def\comp{\hbox{{\rm C}\kern-.55em\hbox{{\rm I} } }}
\def\smallcomp{\hbox{\fiverm C}\kern-.35em{\hbox{\fiverm I}}}

\def\fudge{\mathchoice{}{}{\mkern.5mu}{\mkern.8mu}}
\def\bbc#1#2{{\rm \mkern#2mu\vbar\mkern-#2mu#1}}
\def\bbb#1{{\rm I\mkern-3.5mu #1}} \def\bba#1#2{{\rm #1\mkern-#2mu\fudge
#1}}
\def\bb#1{{\count4=`#1 \advance\count4by-64 \ifcase\count4\or\bba
A{11.5}\or \bbb B\or\bbc C{5}\or\bbb D\or\bbb E\or\bbb F \or\bbc
G{5}\or\bbb H\or \bbb I\or\bbc J{3}\or\bbb K\or\bbb L \or\bbb
M\or\bbb N\or\bbc O{5} \or \bbb P\or\bbc Q{5}\orrrr \bb R\or\bbc
S{4.2}\or\bba T{10.5}\or\bbc U{5}\or    \bba V{12}\or\bba
W{16.5}\or\bba X{11}\or\bba Y{11.7}\or\bba Z{7.5}\fi}}
\def\rat{\hbox{{\rm Q}\kern-.70em\hbox{{\rm I} } }}

\def\BP{{\Bbb P}}
\def\BP{{\Bbb P}}
\def\BR{{\Bbb R}}

\def\BE{{\Bbb E}}
\def\BZ{{\Bbb Z}}

\newcounter{theorem}

\newcounter{theoremcounter}
\newcounter{lemmacounter}
\newcounter{problemcounter}

\newtheorem{theorem}[theoremcounter]{Theorem}
\newtheorem{lemma}[lemmacounter]{Lemma}
\newtheorem{prob}[problemcounter]{Problem}

\newcommand{\eq}{\begin{equation}}
\newcommand{\en}{\end{equation}}
\newcommand{\ignore}[1]{}
\newcommand{\be}{\begin{equation}}
\newcommand{\ber}{\begin{eqnarray}}

\def\proof{\par\noindent{\bf Proof.\enspace}\rm}
\def\blackslug{\hbox{\kern1pt\vrule height6pt width4pt  depth1pt\kern1pt}}
\def\qed{\penalty 500\hbox{\quad\blackslug}\ifmmode\else\par
         \vskip4.5pt plus3pt minus2pt\fi}

\def\bbb#1{{\rm I\mkern-3.5mu #1}} \def\bba#1#2{{\rm #1\mkern-#2mu\fudge
#1}}

\newcommand{\ena}{\end{eqnarray}}

\newtheorem{corollary}{Corollary}
\begin{document} 
%\title{A phase transition for random injections}
\title{Poisson approximations of the number of fixed points in random  multiset permutations}

%\author{Dudley Stark\thanks{Queen Mary, University of London, Mile End, London E1 4NS, UK}}
\author{Dudley Stark\\School of Mathematical Sciences\\Queen Mary, University of London\\ London E1 4NS, UK}
\date{}
\maketitle
\begin{abstract}
Multisets are like sets, except that they can contain multiple copies of their elements. 
%If there are $a_i$ copies of $i$, $1\leq i\leq k$, in multiset $M$, then there are
%$\binom{a_1+\cdots+a_t}{a_1,\ldots, a_t}$ possible permutations (rearrangements)
 %of $M$.
%An element of $M$ is a fixed point of a permutation of $M$ if a copy 
%of that element appears in its position.
For ordinary permutations on $n$ letters, the distribution of the number of fixed points of random permutations is well known to approach the Poisson$(1)$ distribution in total variation distance as $n\to\infty$ super-exponentially quickly.  We use Stein's method to get related results for the number of fixed points of random permutations of multisets.
Given a sequence of multisets on $n$ letters whose expected number of fixed points 
converges to a constant $c$, we must have $c\geq 1$ and the distribution of number
of fixed points converges to the ${\rm Poisson}(c)$ distribution as $n\to\infty$. If $c>1$, then the rate of convergence in total variation distance can be as slow as $n^{-1/2}$.

%We use Stein's method to approximate the distribution of the number of fixed points of multiset permutations by the appropriate Poisson distribution in total variation distance and show Poisson convergence as $n\to\infty$ for multisets on $n$ letters when
%the expected number of fixed points is bounded.
% and Wasserstein distance, respectively. To be effective, both approximations require $\max_{1\leq i\leq t}a_i=o(n)$, in which case the mean and variance of $W$ are both asymptotic to $\sum_{i=1}^t a_i^2/n$. {\bf is this true?}

\end{abstract}

\section{Introduction}

A permutation of the set $[n]:=\{1,\ldots,n\}$ with no fixed points is called a {\em derangement} if it has no fixed points.  Let $D(n)$ denote the number of derangements.
It is well known \cite{Stanley} that 
\eq\label{incexc}
D(n)=
n!\sum_{j=0}^n \frac{(-1)^j}{j}.
\en

Given a permutation chosen uniformly at random from all permutations of the set
 $[n]:=\{1,\ldots,n\}$.
denote the number of fixed points by $W$. A permutation is a derangement if and only if $W=0$.
Given probability measures $\mu_1$ and $\mu_2$ on the set $\BZ_+:=\{0,1,2,\ldots\}$,
The total variation distance between them is defined by
$$
\dtv(\mu_1,\mu_2):=
\sup\{|\mu_1(A)-\mu_2(A)|:A\in\BZ^+\}=
\frac{1}{2}\sum_{j=0}^\infty
|\mu_1(j)-\mu_2(j)|.
$$
In \cite{BHJ}  Example 3.4.2 (The classical recontre or matching problem) it is shown 
using Stein's method that
\eq\label{chen}
\dtv(\calL(W),\Po(1))\leq 2(1-e^{-1})/n,
\en
where Po denotes a Poisson distribution and $\calL(W)$ denotes the probability distribution of $W$. It is also stated that it is easy to show that
\eq\label{fast}
\dtv(\calL(W),\Po(1))\sim\frac{2^n}{(n+1)!},
\en
which shows that the rate of convergence is super-exponentially fast.

Multisets are like sets, except that they can contain multiple copies of their elements. 
Let $\ba=(a_1,\ldots,a_k)$, $a_i\in\BZ^+$ for $1\leq i\leq k$. The multiset $M_\ba$ contains $a_i$ copies of i for all  $1\leq i\leq k(n)$. We also write
$M_\ba=\{1^{a_1},\cdots, k^{a_k}\}$.
As a multiset $[n]$ is defined by $a_i=1$ for all $i\in [n]$.

Define
\eq\label{ndef}
n=\sum_{i=1}^k a_i.
\en
In (\ref{ndef}), $n$ is derived from $\ba$. However, in getting asymptotic results, we will let $n$ be a parameter and let $\ba(n)$ be a sequence of multisets such that
\eq\label{assump}
\ba=\ba(n) {\rm \ satisfies \ } (\ref{ndef}) \ \forall n\geq 1
\en
We will then let
$n\to\infty$. 

A {\em derangment}
of $M_\ba$ is a permutation $\alpha_1\alpha_2\cdots \alpha_n$ (where $n=\sum_{i=1}^n a_i$) of $M_\ba$ that disagrees in each of the $n$ positions with the permutation we get by listing
the elements of $M_\ba$ in weakly increasing order. For instance, the multiset
$\{1,2^2,3\}$ has two derangements: $2132$ and $2312$. Let $D(\ba)$ denote
the number of derangements of $M_\ba$. Problem 2.12 of \cite{Stanley} and its solution
%page 256, soluition page 269
shows that
\eq\label{stanex}
\sum_{\ba\in \BZ_+^k} D(\ba) x^\ba=
\frac{1}{1-\sum_S (|S|-1)\prod_{i\in S}x_i},
\en
where $x^\ba=\prod_{i=1}^k x_i^{a_i}$, $S$ ranges over all nonempty subsets of 
$[n]$, and $|S|$ is the size of $S$. The solution to (\ref{stanex}) starts with the
inclusion-exclusion expression
$$
D(\ba)=
\sum_{b_1=0}^{a_1}\cdots\sum_{b_k=0}^{a_k}
\binom{a_1}{b_1}\cdots\binom{a_k}{b_k}(-1)^{b_1+\cdots+b_k}
\binom{\sum_{i=0}^k (a_i-b_i)}{a_1-b_1,\ldots,a_k-b_k},
$$
which generalises (\ref{incexc}).

%We will call the permutation we get by listing
%the elements of $M_\ba$ in weakly increasing order the {\em original} permutation.
We will say that a {\em fixed point} of a permutation of $M_\ba$ is a position
which agrees with the permutation we get by listing
the elements of $M_\ba$ in weakly increasing order.
%An element of $M_\ba$ is a fixed point of a permutation $\pi$ of $M$ if a copy 
%of that element appears in its position.
For example, the permutation $231121$ of the multiset $\{1^3,2^2,3^1\}$ has fixed points at positions 3 and 5.

Let $W=W(\ba)$ denote the number of fixed points of a permutation chosen uniformly at random from the set $S_\ba$ of all
$\binom{n}{a_1,\ldots,a_k}$ permutations. In this paper, we will derive Poisson approximations for $W$ by using Stein's method.
Thus, we are investigating the recontre or matching problem for multisets. 
We will be able to show Poisson convergence as $n\to\infty$, under assumption (\ref{assump}) on $\ba=\ba(n)$ and assuming that the expected number of fixed points
does not grow too quickly.

Other work has been done on random multiset permutations.
Diaconis and Fulman \cite{DF} study card shuffling on multisets, which amounts to Markov chains on 
multiset permutations.
We consider this work as continuing Stark \cite{Stark}, in which a process generating permutations of multisets was used to find a central limit theorem for the number of 
cycles in a random permutation, where permutation factorisation is defined in
Knuth \cite{Knuth}.

We derive expressions for $\BE W$ and ${\rm Var}(W)$ in Section 2.  Poisson approximations are found in Section 3. Lower bounds
on the poisson approximations are obtained in Section 4, showing that we do not generally have 
the fast convergence seen in (\ref{fast}).

Given sequences of nonnegative numbers $a_n$ and $b_n$, we define $a_n=O(b_n)$ to mean that there is a constant $C>0$ such that $a_n\leq Cb_n$;
we define $a_n=o(b_n)$ to mean that there is a constant $C>0$ such that 
$\lim_{n\to\infty}a_n/b_n=0$;
and we define $a_n\asymp b_n$ to mean that there are constants $C_2>C_1>0$ such that 
$C_1b_n\leq a_n\leq C_2 b_n$.
% and we define $a_n=\Theta(b_n)$ to mean that there is a constant $C_1>0$ such that 
%$a_n\geq C_1 b_n$.

\section{Mean and variance of $W$}\label{expvar}

We will first show that
\begin{lemma}\label{exp}
\eq\label{lambdaexp}
\lambda:=\BE W = \sum_{i=1}^k \frac{a_i^2}{n}.
\en
\end{lemma}
\proof
It is helpful to label the $n$ positions of the original permutation by the elements of the
set 
\eq\label{Gammadef}
\Gamma:=\{(i,j):i\in [k], \, j\in [a_i]\}, 
\en
with the elements of $\Gamma$ in lexicographic order. Thus, the positions $(i,j), j\in[a_i]$ are precisely the positions of the original permutation containing an $i$, for each $i\in [k]$. We define the indicator random variables
$I_{i,j}$ so that 
\eq\label{Iijdef}
I_{i,j}=
\left\{
\begin{array}{l l}
1& {\rm if \ the \ } (i,j) {\rm th \ position \ of \ the \ permutation \ is \ } i\\
0&{\rm otherwise}.
\end{array}
\right.
\en
We have
$$
W=\sum_{i=1}^k\sum_{j=1}^{a_i} I_{i,j}
$$
and also
\eq\label{Iijexp}
\BP(I_{i,j}=1)=\frac{\binom{n-1}{a_1,\ldots,a_{i-1},a_i-1,a_{i+1},\ldots, a_k}}
{\binom{n}{a_1,\ldots, a_k}}=\frac{a_i}{n}.
\en
The expression for $\BE W$ follows immediately.
\qed
Since $a_1^2+a_2^2\leq (a_1+a_2)^2$,  the minimum of $(\ref{lambdaexp})$ is $
\lambda=1$ which occurs when $a_1=1,\ldots,a_n=1$ (the case of  (\ref{chen}) and (\ref{fast}).
The maximum of $(\ref{lambdaexp})$ is $
\lambda=n$ which occurs when $a_1=n$.
Given a real number $x$, let $\lceil x\rceil$ denote the smallest integer greater or equal to 
$x$.
Let $c> 1$ be fixed.
An example satisfying (\ref{assump}) for where $\lambda$ converges to $c$ is
$a_1=\ceil[\Big]{\sqrt{(c-1)n}}, a_2=1, a_3=1,\ldots, a_{n-\ceil[\big]{\sqrt{(c-1)n}}+1}=1$.

The variance of $W$ is given by the next lemma.
\begin{lemma}\label{var}
Assume $n>1$, the case $n=1$ being trivial. Then,
$$
\sigma^2:=\var(W) = \sum_{i=1}^k \frac{a_i^2}{n}\left(1-\frac{a_i}{n}\right)
\left(1-\frac{a_i-1}{n-1}\right)
+\frac{2}{n^2(n-1)}\sum_{1\leq i<j\leq k}a_i^2a_j^2.
$$
\end{lemma}
\proof
We write
\eq\label{varexp}
\var(W)=\sum_{i=1}^k\sum_{j=1}^{a_i}\var(I_{i,j})
+\sum_{i=1}^k \sum_{1\leq j_1<j_2\leq a_i}2\cov(I_{i,j_1}, I_{i,j_2})
+
\sum_{1\leq i_1<i_2\leq k} \sum_{j_1=1}^{a_{i_1}} \sum_{j_2=1}^{a_{i_2}}
2\cov(I_{i_1,j_1}, I_{i_2,j_2})
\en
The first term in (\ref{varexp}) equals
\eq\label{covterms0}
 \sum_{i=1}^k\sum_{j=1}^{a_i}\var(I_{i,j})
=  \sum_{i=1}^k\sum_{j=1}^{a_i}\frac{a_i}{n}\left(1-\frac{a_i}{n}\right)
=  \sum_{i=1}^k \frac{a_i^2}{n}\left(1-\frac{a_i}{n}\right).
\en
We need to calculate the covariances appearing in the second term of (\ref{varexp}).
We have
$$
\BE(I_{i,,j_1}I_{i,,j_2})=\BP(I_{i,,j_1}=1,I_{i,,j_2}=1)=
\frac{\binom{n-2}{a_1,\ldots,a_{i}-2,\ldots, a_k}}
{\binom{n}{a_1,\ldots, a_k}}=\frac{a_i(a_i-1)}{n(n-1)}.
$$
Therefore,
$$
\sum_{i=1}^k \sum_{1\leq j_1<j_2\leq a_i}2\cov(I_{i,j_1}, I_{i,j_2})=
2\sum_{i=1}^k\binom{a_i}{2}\left(\frac{a_i(a_i-1)}{n(n-1)}-\frac{a_i^2}{n^2}\right)
$$
Some algebra shows that
\eq\label{covterms1}
\sum_{i=1}^k \sum_{1\leq j_1<j_2\leq a_i}2\cov(I_{i,j_1}, I_{i,j_2})=
-\sum_{i=1}^k\frac{a_i^2(a_i-1)}{n(n-1)}\left(1-\frac{a_i}{n}\right)
\en
For the last term of (\ref{varexp}), we have
$$
\BE(I_{i_1,,j_1}I_{i_2,,j_2})=\BP(I_{i_1,,j_1}=1,I_{i_2,,j_2}=1)=
\frac{\binom{n-2}{a_1,\ldots,,a_{i_1}-1,\ldots,a_{i_2}-1,\ldots, a_k}}
{\binom{n}{a_1,\ldots, a_k}}=\frac{a_{i_1}a_{i_2}}{n(n-1)},
$$
from which
\begin{eqnarray}\label{covterms2}
\sum_{1\leq i_1<iI_2\leq k} \sum_{j_1=1}^{a_{i_1}} \sum_{j_2=1}^{a_{i_2}}2\cov(I_{_1,j_1},I_{i_2,j_2})&=&
\sum_{1\leq i_1<j_2\leq k} \sum_{j_1=1}^{a_{i_1}} \sum_{j_2=1}^{a_{i_2}}2\left(\frac{a_{i_1}a_{i_2}}{n(n-1)}
-\frac{a_{i_1}a_{i_2}}{n^2}\right)\nonumber\\
&=&
\sum_{1\leq i_1<j_2\leq k} 2\left(\frac{a_{i_1}^2a_{i_2}^2}{n(n-1)}
-\frac{a_{i_1}^2a_{i_2}^2}{n^2}\right)\nonumber\\
&=&\frac{2}{n^2(n-1)}\sum_{1\leq i<j\leq k}a_i^2a_j^2.
\end{eqnarray}
Insertion of (\ref{covterms0}), (\ref{covterms1}) and (\ref{covterms2}) into (\ref{varexp}) and a manipulation
results in the stated formula for $\sigma^2$.
\qed

%We next get a bound on  $\sigma^2$ which will be useful in Section~\ref{lower}.
%\begin{lemma}
%For all $a_1,\ldots,a_k$, $\sum_{i=1}^k=1$, it holds that
%$$
%\sigma^2\leq\lambda-\frac{n}{k^2}+O\left(n^{-1}\right).
%$$
%\end{lemma}
%\proof
%junk
%\qed
%Note that if $k=o(n)$, then for $n$ large enough $\sigma^2<\lambda$.

In the remainder of this section assume (\ref{assump}).

Define
$$
\theta:=\max_{1\leq i\leq k} a_i.
$$
We will use the following inequality often, which follows from By Lemma~\ref{exp}:
\eq\label{thetabound}
\theta\leq\sqrt{n\lambda}
\en

We can show the following.
\begin{lemma}\label{simlam}
Suppose that 
%$\lambda=\Theta(1)$ and 
$\lambda=o(n)$.  
Then,
$
\sigma^2\sim\lambda.
$
\end{lemma}
\proof
It follows from $\lambda=o(n)$ and (\ref{thetabound}) that $\theta=o(n)$.

By Lemma~\ref{var},
\begin{eqnarray}\label{step}
\sigma^2&=&
\sum_{i=1}^k \frac{a_i^2}{n}\left(1-\frac{a_i}{n}\right)
\left(1-\frac{a_i-1}{n-1}\right)
+\frac{2}{n^2(n-1)}\sum_{1\leq i<j\leq k}a_i^2a_j^2\nonumber\\
&=&
\sum_{i=1}^k \frac{a_i^2}{n}\left(1+O\left(\frac{\theta}{n}\right)\right)
+
\frac{1}{n^2(n-1)}O\left(\left(\sum_{i=1}^k a_i^2\right)^2\right)\nonumber\\
&=&
\sum_{i=1}^k \frac{a_i^2}{n}\left(1+O\left(\frac{\theta}{n}\right)\right)
+
\frac{\theta}{n(n-1)}O\left(\sum_{i=1}^k a_i^2\right)\\
&=&
\sum_{i=1}^k \frac{a_i^2}{n}\left(1+O\left(\frac{\theta}{n}\right)\right)\nonumber\\
&\sim&\lambda,\nonumber
\end{eqnarray}
using Lemma~\ref{exp},
where we have used $\sum_{i=1}^k a_i^2\leq\theta\sum_{i=1}^k a_i=\theta n$
at (\ref{step}). 
\qed

If $\lambda\asymp 1$, 
then (\ref{thetabound}) shows that
$
\theta =O\left(\sqrt{n}\right)=o(n)
$
and so, by Lemma~\ref{simlam}, we can expect to have Poisson convergence. In fact,  Corollary~\ref{rate} shows that we have Poisson convergence when 
$\lambda=o(n^{1/3})$.

%If $\lambda=o(n)$ and $\lambda\to\infty$ as $n\to\infty$, then $\sigma^2\sim\lambda$. 
An example for where $\lambda$ grows faster than $o(n^{1/3})$  is
$a_1=\lceil n^{1/3}\rceil, a_2=1, a_3=1,\ldots, a_{n-\lceil n^{1/3}\rceil+1}=1$
for which $\sigma^2\sim\lambda\sim n^{1/3}$.
%It might be expected that if $\sigma^2\to\infty$, then $\calL(W)$ has normal convergence.
Examples of multisets for which $\lambda\asymp n$ are $a_i=c_i n$, $1\leq i\leq k$, 
$k$ and $c_i$ fixed, $\sum_{i=1}^k c_i=1$. We have
$$
\lambda=n\sum_{i=1}^k c_i^2 
$$
and
$$
\sigma^2=n\left(\sum_{i=1}^k c_i^2(1-c_i)^2 
+2\sum_{1\leq i<j\leq k}c_i^2c_j^2\right) +O(1).
$$
%We will not be able to show normal convergence when
%$\sigma^2\to\infty$.

%{\bf We can also get approximations for $W_i$, the number of fixed points restricted to the  copies of $i$.}

\section{Poisson approximations}\label{poisson}

We will now derive a Poisson approximation for $\calL(W)$ by using the framework of \cite{BHJ}, which we will now present. 

Suppose that $W=\sum_{\alpha\in\Gamma} I_\alpha$, where $I_\alpha$ are indicator random variables with expectations $\pi_\alpha$ and $\Gamma$ is an index set. Suppose, moreover, that there exist
random variables $(J_{\beta\alpha}, \beta\in\Gamma)$ on the same probability space as
$(I_\alpha, \alpha\in\Gamma)$ with
\eq\label{coupling}
\calL(J_{\beta\alpha}; \beta\in\Gamma)=
\calL(I_\beta;\beta\in\Gamma|I_\alpha=1).
\en
Suppose that, for each $\alpha\in\Gamma$, the set 
$$
\Gamma_\alpha:=\Gamma\setminus\{\alpha\}
$$
 is partitioned inro
$\Gamma_\alpha^-$, $\Gamma_\alpha^+$, $\Gamma_\alpha^0$ in such a way that
\eq\label{mono}
\begin{array}{ c c c}
J_{\beta\alpha}\geq I_\alpha &{\rm if \ }\beta\in\Gamma_\alpha^+;\\
J_{\beta\alpha}\leq I_\alpha &{\rm if \ }\beta\in\Gamma_\alpha^-;\\
{\rm no \ restriction}&{\rm if \ }\beta\in\Gamma_\alpha^0.
\end{array}
\en
%$J_{\beta\alpha}\leq I_\alpha$ if $\beta\in\Gamma_\alpha^-$;
%and with no condition if $\beta\in\Gamma_\alpha^0$.
Under these assumptions we have the following Poisson approximation theorem.
\begin{theorem}[Theorem 2.C of \cite{BHJ}]\label{SC}
If there exists a coupling satisfying (\ref{coupling}) and (\ref{mono}), then
\begin{eqnarray*}
\dtv(\calL(W),\Po(\lambda)&\leq&\frac{1-e^{-\lambda}}{\lambda}
\left(\sum_{\alpha\in\Gamma} \pi_\alpha^2
+\sum_{\alpha\in\Gamma}\sum_{\beta\in\Gamma_\alpha^-}|\cov(I_\alpha,I_\beta)|\right.\\
&&+ \left.\sum_{\alpha\in\Gamma}\sum_{\beta\in\Gamma_\alpha^+}\cov(I_\alpha,I_\beta)
+ \sum_{\alpha\in\Gamma}\sum_{\beta\in\Gamma_\alpha^0}\left(\pi_\alpha\pi_\beta + \BE(I_\alpha I_\beta\right)\right).
\end{eqnarray*}
\end{theorem}

Using Theorem~\ref{SC}, we obtain
\begin{theorem}\label{mythm}
\eq\label{mybound}
\dtv(\calL(W),\Po(\lambda)\leq\frac{1-e^{-\lambda}}{\lambda}
\left(\sum_{i=1}^k \frac{a_i^3}{n^2}+
\sum_{i=1}^k\frac{a_i^2(a_i-1)}{n(n-1)}\left(1-\frac{a_i}{n}\right)
+\frac{2}{n^2(n-1)}\sum_{1\leq i<j\leq k}a_i^2a_j^2\right),
\en
where $\lambda$ is given by (\ref{lambdaexp}).
\end{theorem}
\proof
The index set $\Gamma$ is given by (\ref{Gammadef}) and the indicators $I_{i,j}$ 
are defined by (\ref{Iijdef}), with 
\eq\label{piijdef}
\pi_{i,j}=a_i/n
\en
 from (\ref{Iijexp}). We must construct the coupling
$J_{\beta\alpha}$. Suppose $\alpha=(i,j)$. If the permutation generating the $I_\alpha$ contains a copy of $i$ in the $(i,j)$th position, then we do nothing. Suppose, instead, the permutation contains a copy of some other number $i^\ast$ in the $(i,j)$th position. We then switch that copy of $i^\ast$ with a randomly chosen copy of $i$. Let us check that this is a proper coupling. Let $\sigma$ be any permutation with $i$ in the $(i,j)$th position. 
The probability $\sigma$ is chosen directly is $\frac{1}
{\binom{n}{a_1,\ldots, a_k}}$. The probability $\sigma$ is chosen after a switching is obtained by choosing a permutation $\sigma^\prime$ obtained by switching the $i$ in the $(i,j)$th position with any copy of any number except for $i$, and then switching that $i$
back to the $(i,j)$th position. The number of $\sigma^\prime$ is $n-a_i$ and the probability of switching back is $1/a_i$. Therefore, the probability of generating $\sigma$ is
$$
\frac{1}
{\binom{n}{a_1,\ldots, a_k}} + \frac{1}
{\binom{n}{a_1,\ldots, a_k}} \frac{n-a_i}{a_i}=
\frac{1}
{\binom{n-1}{a_1,\ldots,a_i-1,\ldots, a_k}},
$$
as is required.
It is easily checked that we may take
\begin{eqnarray*}
\Gamma_{i,j}^+&=&\{(i^\prime, j^\prime)\in\Gamma: i^\prime\neq i\}\\
\Gamma_{i,j}^-&=&\{(i, j^\prime)\in\Gamma: j^\prime\neq j\}\\
\Gamma_{i,j}^0&=&\emptyset.
\end{eqnarray*}
Putting (\ref{covterms1}), (\ref{covterms2}) and (\ref{piijdef}) in Theorem~\ref{SC}
results in the stated bound.
\qed

\begin{corollary}\label{rate}
For any sequence $\ba=\ba(n)$ satisfying (\ref{assump}), 
\eq\label{corbound}
\dtv(\calL(W),\Po(\lambda)=O\left(\sum_{i=1}^k \frac{a_i^3}{n^2}\right)
=O\left(\theta\lambda/n\right)=O(\lambda^{3/2}n^{-1/2}).
\en
\end{corollary}
\proof
%We know $\theta=O(\sqrt{n})$ by (\ref{thetabound}). 
The first two terms in (\ref{mybound}) are bounded by
$$
\sum_{i=1}^k \frac{a_i^3}{n^2}\leq \frac{\theta}{n}\sum_{i=1}^k \frac{a_i^2}{n}
=\frac{\theta\lambda}{n}
$$
and then we apply (\ref{thetabound}).
%We also know that $\sum_{i=1}^k a_i^2=\lambda n$ by (\ref{lambdaexp}).
The third term in (\ref{mybound}) is bounded by
$$
\frac{2}{n^2(n-1)}\sum_{1\leq i<j\leq k}a_i^2a_j^2
\leq\frac{1}{n^2(n-1)}\left(\sum_{i=1}^k a_i^2\right)^2=\frac{\lambda^2}{n-1},
$$
where (\ref{lambdaexp}) has been used for the equality.
%and $\sum_{i=1}^k a_i^3\geq n$ because $a_1^3+a_2^3\leq(a_1+a_2)^3$.
Note that $\lambda^2/(n-1)=o(\lambda^{3/2}n^{-1/2})$ as long as $\lambda=o(n)$.
\qed

When $\lambda=o(n^{2/3})$ the bound (\ref{corbound}) is $o(1)$.
When $\lambda\asymp 1$,
the bound $O(n^{-1/2})$, in (\ref{corbound}) is slower than the rate of convergence, $n^{-1}$,
in (\ref{chen}), which is much slower than the true rate of convergence (\ref{fast}).
That $O(n^{-1/2})$ is the best we can generally obtain from
Theorem~\ref{mythm}  can be seen by considering the examples given by $a_1=\ceil[\Big]{\sqrt{(c-1)n}}, a_2=1, a_3=1,\ldots, a_{n-\ceil[\big]{\sqrt{(c-1)n}}+1}=1$, $c>1$, for which $\lambda=c+O(n^{-1/2})$.
It is immediate that $\sum_{i=1}^k \frac{a_i^3}{n^2}=(c-1)^{3/2}n^{-1/2}+O(n^{-1})$
and so the upper bound on total variation distance obtained from Theorem~\ref{mythm}
is of order
$
n^{-1/2}
$
for these examples.

\section{Lower bounds}\label{lower}

In Example~3.4.2 of \cite{BHJ} it is pointed out that $\lambda=\sigma^2=1$ for permutations
and therefore the lower bounds on total variation distance in Chapter~3 of \cite{BHJ} can not be applied. We will show that one of those lower bounds can often be applied 
in the case of multiset permutations. The theorem on lower bounds we will use follows.

\begin{theorem}[Theorem~3.A of \cite{BHJ}]\label{lowerthm}
Given random variable $W$ taking values in $\BZ_+$ with
$\sigma^2={\rm Var}W<\BE W=\lambda$, we have
\eq\label{lower}
\dtv(\calL(W),\Po(\lambda)\geq
c\varepsilon\left(1+\log(\varepsilon^{-1}\right)^{-1}
\left\{1\wedge(\lambda\vee 1)\left(1+\log(\varepsilon^{-1}\right)^{-1}\right\},
\en
where $c>0$ is a universal constant,
\eq\label{varepsdef}
\varepsilon=(1\wedge\lambda)\left|\frac{\sigma^2}{\lambda}-1\right|,
\en
and $\wedge$ and $\vee$ denote minimum and maximum, respectively.
\end{theorem}

We will apply Theorem~\ref{lowerthm} to our situation.
\begin{theorem}\label{lowermy}
Suppose that (\ref{assump}) holds and
\eq\label{infassump}
\lambda^2=o\left(\frac{1}{n}\sum_{i=1}^k a_i^3\right).
\en
Then, (\ref{lower}) holds for $n$ large enough and 
\eq\label{varepsasym}
\varepsilon\sim (\lambda^{-1}\wedge 1)\frac{2}{n^2}\sum_{i=1}^k a_i^3.
\en
%Moreover, there exists a constant $c^\prime$ such that
%\eq\label{both}
%c^\prime \varepsilon[\log(\varepsilon^{-1})]^{-2}
%\leq\dtv(\calL(W),\Po(\lambda))
%=O(\varepsilon).
%\en
\end{theorem}
\proof
For fixed $k\geq 1$, consider the constrained optimisation problem:
$$
{\rm maximise} 
\sum_{i=1}^k x_i^3
{\rm \ \ subject \ to \ }
\sum_{i=1}^k x_i^2=c, \quad c>0{\rm \ constant}, x_i\in\BR \ \forall i\in [k].
$$
A Lagrange multiplier argument shows that the only extremal solution to this problem
with $x_i>0 \ \forall i\in [k]$ is $\tilde x_i=\sqrt{c/k} \ \forall i\in [k]$, for which
$\sum_{i=1}^k \tilde x_i^3=k^{-1/2}c^{3/2}$. Thus, $\sum_{i=1}^k x_i^3\leq c^{3/2}$ for all $k$ and all $x_i$ satisfying $\sum_{i=1}^k x_i^2=c$. Therefore, if $\lambda=\frac{1}{n}\sum_{i=1}^k a_i^2=cn$
for $c\in(0,1]$, then
$
\frac{1}{n}\sum_{i=1}^k a_i^3\leq c^{3/2} n^2.
$
As a result of $c^{3/2}\leq c^2$, if (\ref{infassump}) holds, then it must be the case that
$\lambda=o(n)$.

We know that $\theta=o(n)$ from (\ref{thetabound}) and $\lambda=o(n)$.
By Theorem~\ref{exp} and Theorem~\ref{var}, we have
\begin{eqnarray}\label{explan}
\lambda-\sigma^2&=&\sum_{i=1}^k\frac{a_i^3}{n^2}
+\sum_{i=1}^k\frac{a_i^2(a_i-1)}{n(n-1)}
-\sum_{i=1}^k\frac{a_i^3(a_i-1)}{n^2(n-1)}
-\frac{2}{n^2(n-1)}\sum_{1\leq i<j\leq k}a_i^2a_j^2\nonumber\\
&\geq&
2\left(1-\frac{\theta}{n}\right)\sum_{i=1}^k\frac{a_i^3}{n^2}
-\sum_{i=1}^k\frac{(n-a_i)a_i^2}{n^2(n-1)}
-\frac{2}{n^2(n-1)}\left(\sum_{i=1}^ka_i^2\right)^2\nonumber\\
&\geq&
2(1-o(1))\sum_{i=1}^k\frac{a_i^3}{n^2}
-\frac{\lambda}{n}
-\frac{2\lambda^2}{n-1}\nonumber\\
&=&
2(1-o(1))\sum_{i=1}^k\frac{a_i^3}{n^2},
\end{eqnarray}
where (\ref{infassump}) has been used at (\ref{explan}).
Now, (\ref{varepsasym}) follows immediately from (\ref{varepsdef}).
By (\ref{explan}), $\sigma^2<\lambda$ for large enough $n$ and (\ref{lower}) 
is applicable.
\qed

%We will apply Theorem~\ref{lowermy} to an example, showing that we do not generally have convergence as fast as in (\ref{fast}).
Let $a_1=\ceil[\Big]{\sqrt{(c-1)n}}, a_2=1, a_3=1,\ldots, a_{n-\ceil[\big]{\sqrt{(c-1)n}}+1}=1$,
$c>1$. We saw in
Section~\ref{poisson} that the upper bound in (\ref{corbound})
is of order $O(n^{-1/2})$. Clearly, $\lambda=o(n)$, (\ref{infassump}) holds
and (\ref{varepsasym}) gives us $\varepsilon\sim 2(c^{-1}\wedge 1)(c-1)^{3/2}n^{-1/2}$. Thus, the lower bound (\ref{lower}) is of order
$c^\prime n^{-1/2}/\log^2 n$ for some constant $c^\prime>0$.

%We leave unresolved the question as to whether $\sigma^2<\lambda$ whenever $a_i>1$ for some $i\in[k]$. 

\end{document}